\documentclass[11pt,reqno]{preprint}

\usepackage[full]{textcomp}
\usepackage[osf]{newtxtext}
\usepackage{comment}

\usepackage{amssymb}
\usepackage{mathtools}
\usepackage{hyperref}
\usepackage{breakurl}
\usepackage{mhenvs}
\usepackage{mhequ}
\usepackage{mhsymb}
\usepackage{booktabs}
\usepackage{tikz}
\usetikzlibrary{decorations.pathreplacing}
\usetikzlibrary{decorations.markings}
\usetikzlibrary{decorations.pathmorphing}
\usetikzlibrary{snakes}
\usepackage{mathrsfs}
\usepackage{longtable}
\usepackage{scalerel}
\usepackage{cprotect}
\usepackage{multirow}
\usepackage{colortbl}
\usepackage{makecell}
\usepackage{resizegather}
\usepackage{enumerate}

\def\cP{{\mathcal P}}

\newcommand{\mfg}{\mathfrak{g}}

\def\bt{\begin{theorem}}
	\def\et{\end{theorem}}
\def\bl{\begin{lemma}}
	\def\el{\end{lemma}}
\def\br{\begin{remark}}
	\def\er{\end{remark}}
\def\bx{\begin{Examples}}
	\def\ex{\end{Examples}}
\def\bd{\begin{definition}}
	\def\ed{\end{definition}}
\def\bp{\begin{proposition}}
	\def\ep{\end{proposition}}
\def\bc{\begin{corollary}}
	\def\ec{\end{corollary}}
\def\<{{\langle}}
\def\>{{\rangle}}

\newcommand{\dif}{{\mathord{{\rm d}}}}
\newcommand{\Tr}{\operatorname{Tr}}
\def\Re{\operatorname{Re}}
\newcommand{\Var}{\operatorname{Var}}

\newcommand{\Ent}{\operatorname{Ent}}
\newcommand{\Hess}{\operatorname{Hess}}
\newcommand{\osc}{\operatorname{osc}}
\newcommand{\Ric}{\operatorname{Ric}}

\definecolor{codegreen}{rgb}{0,0.6,0}
\definecolor{codegray}{rgb}{0.5,0.5,0.5}
\definecolor{codepurple}{rgb}{0.58,0,0.82}
\definecolor{backcolour}{rgb}{0.95,0.95,0.92}

\tikzset{
    midarrow/.style={postaction={decorate,decoration={markings,mark=at position 0.7 with {\arrow{>}}}}},
    midarrow1/.style={postaction={decorate,decoration={markings,mark=at position 0.45 with {\arrow{>}}}}},
    midarrow2/.style={postaction={decorate,decoration={markings,mark=at position 0.75 with {\arrow{>}}}}},
}

\colorlet{darkblue}{blue!90!black}
\colorlet{darkred}{red!90!black}
\colorlet{darkgreen}{green!70!black}

\title{Log-Sobolev inequality for 2D $\mathrm{SU}(N)$ lattice gauge theory at all couplings}

\author{Jack Piazza$^1$  and Hao Shen$^2$}

\institute{University of Wisconsin--Madison, \email{japiazza2@wisc.edu} \and 
University of Chicago, \email{pkushenhao@gmail.com}}

\begin{document}

\maketitle

\begin{abstract}
We prove that the log-Sobolev inequality holds for the two-dimensional $SU(N)$ Wilson lattice Yang--Mills measure with arbitrary fixed inverse coupling $\beta$. 
The log-Sobolev constant depends on $\beta$ and $N$, but is uniform in the volume of the underlying lattice. The proof relies on a two-scale disintegration argument.
\end{abstract}

\setcounter{tocdepth}{2}

\section{Introduction}

For $\beta\in\R$ and $N \geq 2$, consider the $G=\mathrm{SU}(N)$ Wilson lattice Yang--Mills measure on $\Lambda_L=(\Z/L\Z)^2$:
\begin{equ}\label{eq:YMmeasure}
    \mu_{L,N,\beta}(\dif U)
    =\frac1{Z_{L,N,\beta}}
    \exp\bigg\{N\beta\sum_{p\in\cP^+_L}\Re\Tr(U_p)\bigg\}
    \prod_{e\in E^+_L}\dif U_e
\end{equ}
where $\dif U$ is the normalized Haar probability measure on $G$ and $E^+_L$ and $\cP^+_L$ are the sets of positively oriented edges and plaquettes in $\Lambda_L$, respectively. We refer the reader to \cite{Chatterjee18} for an introduction to this measure and the relevant terminology.

\begin{theorem}\label{thm:main}
For every integer $N\ge2$ and every $\beta\in\R$ there exists $C_{N,\beta}>0$, depending only on $N$ and $\beta$, such that for every $L\ge1$ and every smooth $F:G^{E^+_L}\to\R$,
\begin{equ}\label{eq:mainLSI}
    \Ent_{\mu_{L,N,\beta}}(F^2)
    \le 2C_{N,\beta}
    \sum_{e\in E^+_L}\int |\nabla_eF|^2\,\dif\mu_{L,N,\beta}.
\end{equ}
\end{theorem}

Here we recall that 
the entropy of a probability measure $\nu$ is defined as
\[
    \Ent_\nu(f^2)=\nu(f^2\log f^2)-\nu(f^2)\log\nu(f^2)
\]
with the convention $0\log0=0$. The measure satisfies a log-Sobolev inequality (LSI) if there is a constant $C > 0$ such that $\Ent_\nu(f^2)\le2C_{\mathrm{LS}}(\nu)\nu(|\nabla f|^2)$. for all smooth $f$, and denote by $C_{\mathrm{LS}}(\nu)$ the optimal constant in this inequality.

Log-Sobolev inequalities were introduced in \cite{MR0420249}, and have long been studied in the context of statistical mechanics models due to their connection to hypercontractivity and exponential relaxation for the associated Glauber or Langevin dynamics. See \cite{MR3155209, BBD24, GZ, MR1837286, zegarlinski1990log} for some background and foundational results on log-Sobolev inequalities and related functional inequalities arising in lattice models.

Poincar\'e and log-Sobolev inequalities for lattice Yang--Mills measures have been obtained in all spatial dimensions in \cite{SZZ22}, but only in the strong coupling regime where $\lvert \beta \rvert$ is sufficiently small. The current paper extends this to all $\beta \in \R$ in the two-dimensional case, where the geometry of the model is substantially simpler.
We remark that the Yang--Mills measure on $\R^2$ has decorrelated Wilson loops, namely the existence of a positive mass gap is essentially trivial in 2D. However LSI is a much stronger condition for a measure (c.f. \cite{SZZ22} which shows that one can derive the mass gap from LSI in the context of lattice Yang--Mills with small $\beta$).

Our argument relies on a spatial two-scale decomposition of the field. There is much prior literature using such spatial decompositions to prove log-Sobolev inequalities for other lattice models, see for example \cite{RolandSG,RolandPhi4,RolandIsing}. Our proof is motivated by the works \cite{OttoReznikoff2007,GOVW2009}, and has much in common with the simple approach taken in \cite{MR3926125}. A somewhat similar approach was used to study the lattice Yang--Mills--Higgs model in \cite{SZZ2024Higgs}, though there the decomposition is over different fields rather than different spatial scales, and the result is again only valid at strong coupling.

The special feature of the two-dimensional Yang--Mills model is that the small-scale integrations can be done by convolutions of probability weights on the Lie group (see \cite{Chevyrev19YM,MR5084118,MR4876989} for similar ideas). Moreover, in two dimensions, conditioning on the boundaries of the blocks (which are the ``large-scale'' variables),  the small-scale variables are independent across different blocks. In higher dimensions, conditional integrations would naturally produce nonlocal, larger Wilson loops, which would be much more difficult to analyze.

\begin{remark}
	The proof strategy we employ also works to prove an analogous result for $G = \mathrm{SO}(N)$ with $N \geq 3$. However, since we  rely on $G$ having uniform positive Ricci curvature, our method does not extend to $G=\mathrm{U}(1)$ (or equivalently, $G=\mathrm{SO}(2)$) without further ideas.
\end{remark}

\noindent \textbf{Acknowledgements:}
H.S. gratefully acknowledges financial support from the NSF grant CAREER DMS-2044415, and would like to thank Roland Bauerschmidt, who once mentioned the possibility of applying a one-step RG argument to 2D lattice Yang--Mills. We also thank Kevin Hu for many discussions on RG related problems.

\medskip

\noindent \textbf{AI disclosure:} Claude (Opus and Sonnet 5.5) was used for literature search and to verify and find references for some geometric formulas used in the proofs of Lemmas \ref{lem:C2decay} and \ref{lem:HessHbar}. The ideas, arguments, and writing in this paper are due to the authors.

\section{Convolution of the single-plaquette weight}

In our proof, we will divide the two-dimensional lattice into blocks, then perform conditional integrations inside each block. The resulting measures we obtain will be given by convolutions of weights associated to individual plaquettes.
The key observation is that convolving these weights makes the associated Hessian smaller.
Writing
\[
    w(U)=\frac{e^{N^2\beta\chi(U)}}{z_{N,\beta}},
    \qquad
     \chi(U)=\frac1N \Re \Tr (U),
    \qquad
    z_{N,\beta}=\int_G e^{N^2\beta\chi(V)}\,\dif V
\]
then $w$ is a  probability density, and since $\chi(U)\in [-1,1]$,
\begin{equ}[eq:bound-w]
    e^{-2t}\le w\le e^{2t} ,\qquad   t=N^2|\beta|.
\end{equ}
Let $w^{*k}$ denote the $k$-fold convolution of $w$, where
\[
    (f*g)(x)=\int_G f(y)g(y^{-1}x)\,\dif y.
\]
Recall that if $g$ is a probability density then $\| f*g\|_\infty \leq \|f\|_\infty$ for any bounded $f$. Also, the convolution of two probability densities is again a probability density, and in particular $1*g = 1$.

First, we show that $w^{*k}$ is asymptotically $1$ as $k\to \infty$. This is essentially because the random walk on $G$ with transition probability $w$ converges to being Haar distributed after many steps.
Write $r=1-\frac12 e^{-2t}$ and\footnote{The parameter $r\in (0,1)$ plays a key role in the contraction of the Hessian; note that $r\to 1$ as $\beta\to \infty$, so this contractive effect is weaker for larger $\beta$.} 
\[
w=(1-r) + r v ,\qquad v=(w-\frac12 e^{-2t})/r,
\]
noting that $v$ is also a probability density.

\begin{lemma}\label{lem:convdecomp}
For all $k \geq 2$,
\begin{equ}\label{eq:convdecomp}
    w^{*k} = 1+r^{k-1}(w-1)*v^{*(k-1)}.
\end{equ} 
\end{lemma}

\begin{proof}
We can write
\[
w^{*k} = w^{*(k-1)}*w = w^{*(k-1)}*((1-r) + r v )
=(1-r) + rw^{*(k-1)}*v.
\]
Therefore,
\[
w^{*k} - 1 = r (w^{*(k-1)}*v-1*v) = r (w^{*(k-1)}-1)*v.
\]
This gives \eqref{eq:convdecomp} for $k=2$, and the full claim follows by induction.
\end{proof}

Let $\mathfrak{g} = \mathfrak{su}(N)$ be the Lie algebra of $G$, equipped with the Hilbert--Schmidt inner product \footnote{We write $\dagger$ here for conjugate transpose  instead of $*$ since  $*$ is reserved for convolution.} $\langle X, Y \rangle = \Tr (X^\dagger Y)$.
We also write 
\[
\|\Hess f \|_{\infty} 
= \sup_{U\in H} \|\Hess f(U)\|_{\operatorname{op}}
= \sup_{U\in H}\sup_{\| X \|=1} |\Hess f(U)(X,X)|
\]
where $H$ is the Lie group on which $f$ is defined, which will always be either $G$ or some power of $G$, and $X$ is interpreted in expressions like $\Hess f(U) (X, X)$ as the left-invariant vector field associated to $X \in \mathfrak{g}$.

\begin{lemma}\label{lem:C2decay}
There exists $M$ (explicitly depending on $N$ and $\lvert \beta \rvert$) so that
\begin{equ}\label{eq:C2decay}
    \max\left\{
    \|w^{*k}-1\|_\infty,
    \| \nabla w^{*k} \|_\infty,
    \|\Hess w^{*k}\|_\infty
    \right\}
    \le Mr^{k-1}.
\end{equ}
\end{lemma}

\begin{proof}
By Lemma~\ref{lem:convdecomp} and \eqref{eq:bound-w}
\[
\|w^{*k}-1\|_\infty \le r^{k-1} \|w-1\|_\infty 
\le r^{k-1} e^{2t}.
\]
Let $X$ be a unit vector in $\mathfrak{g}$. Then
\begin{equ}\label{eq:gradW}
	\nabla_X w (U) = w(U)  N \beta \Re \Tr (UX).
\end{equ}
By Cauchy--Schwarz, \[
	\lvert \Re \Tr (UX) \rvert \leq \| U^\dagger \| = \sqrt{N}
\]
so we have the uniform bound $\lvert \nabla_X w (U) \rvert \leq w(U) N^{3/2} \lvert \beta \rvert$ and thus, combining with \eqref{eq:bound-w}, $\| \nabla w \|_\infty \leq e^{2t} N^{3/2} \lvert \beta \rvert$.
For the convolution,
\[
	\nabla_X w^{*k} = \nabla_X ((1-r) + r w^{*(k-1)} * v) = rv*(\nabla_X (w^{*(k-1)}))
\]
so 
\[
	\| \nabla w^{*k} \|_\infty \leq r \| \nabla w^{*(k-1)}\|_\infty .
\]
By induction with the bound on $\| \nabla w \|_\infty$ as the base case, we obtain $\| \nabla w^{*k} \|_\infty \leq r^{k-1} e^{2t} N^{3/2} \lvert \beta \rvert$.

For the Hessian, we use the general formula 
\[
	\Hess w (U) (X, Y) = \nabla_X \nabla_Y w(U) - \frac{1}{2} \nabla_{[X, Y]} w(U)
\]
so in particular $\Hess w(U) (X, X) = \nabla_X \nabla_X w(U)$. Taking $X$ to be a unit vector in $\mathfrak{g}$, we compute
\[
	\nabla_X \nabla_X w(U) = w(U) N^2 \beta^2 (\Re \Tr (UX))^2 + w(U) N \beta \Re \Tr (U X^2) .
\]
Using the Cauchy--Schwarz bounds $\lvert \Re \Tr (UX) \rvert \leq \sqrt{N}$ and $\lvert \Re \Tr (U X^2) \rvert \leq 1$ (where the latter is obtained by decomposing $UX^2$ as $UX$ times $X$), we obtain
\[
	\lvert \Hess w(U) (X, X) \rvert \leq w(U) (N^3 \lvert \beta \rvert^2 + N \lvert \beta \rvert)
\]
and thus $\| \Hess w \|_\infty \leq e^{2t} (N^3 \lvert \beta \rvert^2 + N \lvert \beta \rvert)$. As with the gradient we have
\[
	\| \Hess w^{*k} \|_\infty \leq r \| \Hess w^{*(k-1)} \|_\infty
\]
so iterating yields $\| \Hess w^{*k} \|_\infty \leq r^{k-1} e^{2t} (N^3 \lvert \beta \rvert^2 + N \lvert \beta \rvert)$.
\end{proof}

For sufficiently large $n$ such that $Mr^{n-1} \leq 1/3$, we have $w^{*k} \geq 2/3$ for all $k \geq n$. Then for all $k \geq n$, Lemma \ref{lem:C2decay} implies
\begin{equ}\label{eq:logwk}
    \|\nabla\log w^{*k}\|_\infty
    +\|\Hess\log w^{*k}\|_\infty
    \le 4M r^{n-1}.
\end{equ}
Here we used the formula $\Hess \log f = (\Hess f)/f - (df\otimes df)/f^2$.

\section{Disintegration}

Let $B$ be an $m \times m$ block of the lattice and $n=m^2$. We first study the portion $\prod_{p\subset B} w(U_p)$ of the Yang--Mills measure coming from plaquettes in $B$.

Consider an ``axial'' spanning tree over the centers of all the plaquettes $p\subset B$, as on the left diagram of Figure \ref{fig:axial-tree}. Let $I_B$ be those edges which cross the spanning tree. We only integrate the edges in $I_B$,
which is a convolution
\begin{equ}\label{eq:blockconvolution}
    \int_{G^{I_B}}\prod_{p\subset B} w(U_p)\,\prod_{e\in I_B}\dif U_e
    =w^{*n}(W_B(s))
\end{equ}
where $W_B$ is a word over $G$ made from the edge variables
\[
s\in S_B \eqdef I_B^c = \{e: \mbox{$e$ belongs to a plaquette in $B$}\} \setminus I_B.
\]
See the right side of Figure \ref{fig:axial-tree} for an idea of what the edge sets $I_B$ and $S_B$ look like.
For example, if $I_B$ contained just one edge the integral would look like
\[
\int_G w(AuB) w(Cu^{-1}D) \dif u = w^{*2} (BADC).
\]
Before any reduction, $W_B$ has $2n+2$ letters. Indeed, the density originally contains $4n$ occurrences of edge variables (four for each plaquette); there are $n-1$ edges in $I_B$, one for each edge of the spanning tree, and each integration of an edge in $I_B$ removes two letters. 

In fact, the word $W_B$ has many redundancies, and can be reduced to the length $4\sqrt{n}$ word $W_{\partial B} (s)$ tracing the boundary $\partial B$ of the block by canceling adjacent occurrences of edges and their inverses. We thus take $W_B (s) = W_{\partial B} (s)$ from here on.

\begin{remark}
	In the above, one could take any spanning tree over the centers of the plaquettes in $B$ rather than the axial spanning tree. We choose the axial tree because it is clear in that case that one has $W_B (s) = W_{\partial B} (s)$, but one can show that this claim holds for any spanning tree with an elementary geometric argument.
\end{remark}

\begin{figure}
\centering
\begin{tikzpicture}
	\foreach \i in {0,...,4} {
        \draw[gray!60, thin] (\i, 0) -- (\i, 4);
        \draw[gray!60, thin] (0, \i) -- (4, \i);
    }
    \draw[black, thick] (.5, 3.5) -- (3.5, 3.5);
    \draw[black, thick] (.5, .5) -- (.5, 3.5);
    \draw[black, thick] (1.5, .5) -- (1.5, 3.5);
    \draw[black, thick] (2.5, .5) -- (2.5, 3.5);
    \draw[black, thick] (3.5, .5) -- (3.5, 3.5);
\end{tikzpicture}
\qquad\qquad
\begin{tikzpicture}
	\foreach \i in {0,...,4} {
        \draw[gray!60, thin] (\i, 0) -- (\i, 4);
        \draw[gray!60, thin] (0, \i) -- (4, \i);
    }
    
    \draw[red, dotted, thick] (0, 1) -- (4, 1);
    \draw[red, dotted, thick] (0, 2) -- (4, 2);
    \draw[red, dotted, thick] (0, 3) -- (4, 3);
    \draw[red, dotted, thick] (1, 3) -- (1, 4);
    \draw[red, dotted, thick] (2, 3) -- (2, 4);
    \draw[red, dotted, thick] (3, 3) -- (3, 4);
    
    \draw[blue, thick] (0, 0) -- (0, 4) -- (4, 4) -- (4, 0) -- (0, 0);
    \draw[blue, thick] (1, 0) -- (1, 3);
    \draw[blue, thick] (2, 0) -- (2, 3);
    \draw[blue, thick] (3, 0) -- (3, 3);
\end{tikzpicture}
\caption{Left: the axial spanning tree over the centers of plaquettes in the block. Right: the edges in $I_B$ (dotted red) and in $S_B$ (solid blue). The $2n+2$-letter word formed by the blue edges reduces to the word $W_{\partial B}$, since all of the vertical ``spikes'' on the interior of the block correspond to subwords which reduce to the identity.}
\label{fig:axial-tree}
\end{figure}
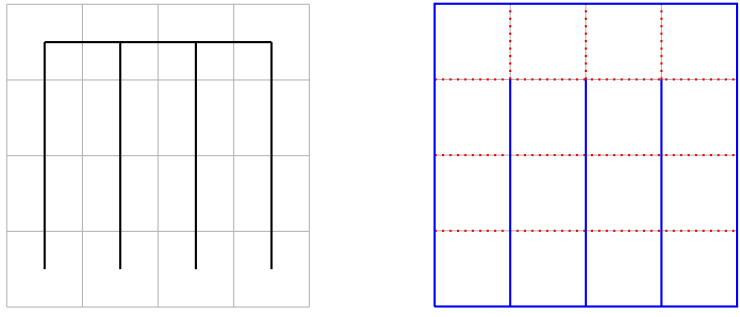

Now we disintegrate. We pave our lattice $\Lambda_L$ by blocks of size $n=m\times m$.\footnote{Later, we will fix a large enough $m$. If $L \geq 2m$ is not a multiple of $m
$, then we write $L=m+\cdots+m+k$ with $k\in [m,2m-1)$ and pave $\Lambda_L$ by blocks of size $m\times m$ except for one row and column. For the remainder, we use blocks of size $k\times m$, $m\times k$, or $k\times k$, and all our bounds obviously hold on these blocks as well. If $L < 2m$ (or more generally for any finitely many small values of $L$) the blocking procedure is unnecessary, and the log-Sobolev inequality follows directly from the Holley--Stroock criterion as in Section \ref{sec:conditional-LSI}.}
Define
\[
I=\bigcup_{B}I_B,\qquad S= \bigcup_B S_B = I^c
\]
so that $U=((U_e)_{e\in I}, (U_e)_{e\in S})=: (x,s)$.
Then
\[
\mu(\dif s,\dif x)=\bar\mu(\dif s) \nu_s(\dif x)
\]
where the  marginal law $\bar \mu$ of the remaining edges is
\begin{equ}\label{eq:marginal}
    \bar\mu(\dif s)
    =\frac1{\bar Z}
    \prod_{B} w^{*n}(W_{\partial B}(s))\, \dif s
\end{equ}
and the conditional law $\nu_s$ is
\begin{equ}\label{eq:conditional-blocks}
    \nu_s(\dif x)=\bigotimes_{B}\nu_{B,s}(\dif x_B),\qquad \nu_{B,s}(\dif  x_B)
    =\frac{\prod_{p\subset B}w(U_p)}{w^{*n}(W_{\partial B}(s))}\,\dif x_B.
\end{equ}

\section{Log-Sobolev inequality for the marginal law}

Let
\[
\bar\mu(ds)=\bar Z^{-1}e^{\bar H(s)} ds,
\qquad
    \bar H(s)=\sum_{B}\log w^{*n}(W_{\partial B}(s)).
\]
Let $\psi (g) = \log w^{*n}(g)$ so that $\bar H=\sum_B \psi\circ W_{\partial B} = \sum_B \bar{H}_B$. We prove the following bound on the Hessian.

\begin{lemma}\label{lem:HessHbar}
	The following inequality holds:
	\[
		\| \Hess \bar{H}_B \|_\infty \leq 4 \sqrt{n} (\| \nabla \psi \|_\infty + \| \Hess \psi \|_\infty) .
	\]	
	Consequently,
	\[
		\| \Hess \bar{H} \|_\infty \leq 8 \sqrt{n} (\| \nabla \psi \|_\infty + \| \Hess \psi \|_\infty) .
	\]	
\end{lemma}
\begin{proof}
Fix a block $B$. Let $P = 4\sqrt{n}$ be the length of $W_{\partial B}$ (the perimeter of the block), and write $W$ for $W_{\partial B}$. Note that $W$ is just a product
\[
	W(U_1, \dots, U_P) = W_{U_1} \dots W_{U_P} .
\]
Recall that the Hessian of $\psi \circ W$ can be defined via
\[
	\Hess (\psi \circ W) (X, X) = \frac{d^2}{dt^2} (\psi (W (\gamma (t)))) \Big|_{t=0}
\]
where $\gamma$ is a geodesic on $G^P$ with $\gamma (0) = U$ and $\dot{\gamma} (0) = X$ (see \cite[Chapter 5]{AMS08} for an introduction to the relevant notions). More precisely, we take $\gamma (t) = (U_1 e^{tX_1}, \dots, U_P e^{tX_P})$. Letting $c(t) = W(\gamma (t))$, this becomes
\begin{equ}\label{eqn:hess-geodesic}
	\Hess (\psi \circ W) (X, X) = \Hess \psi (\dot{c} (0), \dot{c} (0)) + \langle \nabla \psi, \nabla_{\dot{c} (0)} \dot{c} (0) \rangle .
\end{equ}
Note that $c$ is not a geodesic on $G$, so the second term is nonzero in general.

To bound the first term, we need to find $\Hess \psi (dW(X), dW(X))$. First we compute $dW(X)$ with $X = (X_1, \dots, X_P)$. This gives
\begin{equs}\label{eq:dW}
	dW(X) &= \frac{d}{dt} \Big|_{t=0} U_1 e^{tX_1} \dots U_P e^{tX_P} = \frac{d}{dt} \Big|_{t=0} U_1 U_2 \dots U_P e^{tZ_1} \dots e^{tZ_P} 
	\\
	&= U_1 U_2 \dots U_P (Z_1 + \dots + Z_P)
\end{equs}
where $Z_P = X_P$, $Z_{P-1} = U_P^{-1} X_{P-1} U_P$, and so on, down to $Z_1 = (U_2 \dots U_P)^{-1} X_1 (U_2 \dots U_P)$. Therefore,
\[
	\lvert \Hess \psi (dW(X), dW(X)) \rvert \leq \| \Hess \psi \|_\infty \| Z_1 + \dots + Z_P \|^2 \leq P \| \Hess \psi \|_\infty \| X \|^2 .
\]

For the second term, we must compute $\nabla_{\dot{c} (0)} \dot{c} (0)$. First, we write $c(t)$ as $U_1 \dots U_P \exp (z(t))$ for some curve $z(t)$. We have by the Baker--Campbell--Hausdorff formula that 
\[
	c(t) = U_1 e^{tX_1} \dots U_P e^{tX_P} = U_1 \dots U_P e^{tZ_1} \dots e^{t Z_P} = U_1 \dots U_P \exp (z(t))
\]
where
\[
	z(t) = t \sum_i Z_i + \frac{t^2}{2} \sum_{i<j} [Z_i, Z_j] + O(t^3).
\]
By \cite[Theorem 5.4]{Hall15}, this implies
\[
	\dot{c}(t) = U_1 \dots U_P e^{z(t)} \Big( z'(t) - \frac{[z(t), z'(t)]}{2!} + \frac{[z(t), [z(t), z'(t)]]}{3!} + \dots \Big) =: U_1 \dots U_P e^{z(t)} Y(t)
\]
and in particular $\dot{c} (0) = U_1 \dots U_P z'(0)$. Therefore,
\begin{equs}
	\nabla_{\dot{c}(0)} \dot{c}(0) = U_1 \dots U_P e^{z(0)} Y'(0) + \frac{1}{2} U_1 \dots U_P [z'(0), Y(0)].
\end{equs}
One can compute from the definition that $Y(t) = z'(t) + O(t^2)$, so
\begin{equ}
	\nabla_{\dot{c}(0)} \dot{c}(0) = U_1 \dots U_P z''(0) = U_1 \dots U_P \sum_{i < j} [Z_i, Z_j].
\end{equ}
Thus, the second term of \eqref{eqn:hess-geodesic} is bounded by 
\[
	\lvert \langle \nabla \psi, \nabla_{\dot{c} (0)} \dot{c} (0) \rangle \rvert \leq \| \nabla \psi \|_\infty \sum_{i < j} \| [Z_i, Z_j] \| \leq P \| \nabla \psi \|_\infty \| X \|^2 .
\]
Here we used 
\[
\sum_{i<j}\|[Z_i,Z_j]\|
\le 2 \sum_{i<j}\|Z_i\|\|Z_j\|
\le P \sum_i \|Z_i\|^2.
\]
From this and \eqref{eqn:hess-geodesic} we can conclude the bound on $\Hess \bar{H}_B$.

The bound on $\Hess \bar{H}$ follows immediately, since $\bar{H} = \sum_B \bar{H}_B$ and each edge is contained in at most two blocks. Indeed, one has
\[
	\lvert \Hess \bar{H} (X, X) \rvert =\Big\lvert \sum_B \Hess \bar{H}_B (X, X) \Big\rvert = \Big\lvert \sum_B \Hess \bar{H}_B (X_B, X_B) \Big\rvert \leq C \sum_B \|X_B \|^2
\]
where $C$ is the upper bound from the first part of the lemma, and $X_B$ consists of only components of $X$ in the directions of edges in $B$. But since each edge occurs in at most two blocks, this last sum is at most $2 \sum_e \| X_e \|^2 = 2 \| X \|^2$.
\end{proof}

Now we apply \eqref{eq:logwk} to the right-hand side of the bound in Lemma~\ref{lem:HessHbar}. If $n$ is large enough,
\[
|\Hess \bar H(X,X)|\le \frac12 \|X\|^2.
\]
Note that unlike the original Yang-Mills action whose Hessian bound would be proportional to $\beta$, the bound here does not depend on $\beta$.
On the other hand, $\Ric_G=\frac{N}2 g$ where $g$ is the metric. So
\[
\Ric_G - \Hess \bar H \ge \frac12 g.
\]
Hence, by the Bakry--\'Emery criterion~\cite{BE},
\begin{equ}\label{eq:marginalLSI}
    \Ent_{\bar\mu}(f^2)
    \le 4 \sum_{e\in S} \bar\mu (|\nabla_e f|^2 ).
\end{equ}

\section{Uniform log-Sobolev inequality for the conditional laws}\label{sec:conditional-LSI}

For the conditional laws $\nu_s$, we use the Holley--Stroock criterion~\cite{HolleyStroock1987}:
 If $\nu$ satisfies an LSI with constant $C$ and
\[
    \dif \widetilde\nu=Z^{-1}e^V\,\dif \nu,
\]
then
\begin{equ}\label{eq:HSperturb}
    C_{\mathrm{LS}}(\widetilde\nu)
    \le e^{\osc V} C_{\mathrm{LS}}(\nu),
    \qquad \osc V:=\sup V-\inf V.
\end{equ}

By the Bakry--\'Emery criterion,
a finite Haar product $\prod_e \dif U_e$ satisfies
\begin{equ}\label{eq:HaarLSI}
    \Ent(f^2)
    \le \frac4N\sum_{j}\int |\nabla_jf|^2\,\prod_e \dif U_e
    \le 2\sum_{j}\int |\nabla_jf|^2\,\prod_e \dif U_e.
\end{equ}

For a block $B$, set 
\[
    H_B(s_B ,x_B)=\sum_{p\subset B}\log w(U_p).
\]
Since each plaquette contributes an oscillation at most $4t$, we have
\begin{equ}\label{eq:block-oscillation}
    \osc_{x_B} H_B(s_B,\cdot)\le 4tn.
\end{equ}

Therefore,
\begin{equ}\label{eq:conditionalLSI}
    \Ent_{\nu_s}(F^2)
    \le 2K\,\nu_s\Big(\sum_{i\in I}|\nabla_iF|^2\Big),
    \qquad
     K=e^{4tn}
\end{equ}
which follows for each block measure $\nu_{B, s}$ by the Holley--Stroock criterion,
and then for the full measure $\nu_s$ by tensorization (c.f. \cite{GOVW2009}). 
Note in particular that a log-Sobolev inequality implies a Poincar\'e inequality with the same constant, so for each block we have
\begin{equ}[eq:blockPoincare]
    \Var_{\nu_{B,s}}(f)
    \le K\,\nu_{B,s}\Big(\sum_{i\in I_B}|\nabla_if|^2\Big).
\end{equ}

\section{Proof of the main result}

Recall that $\mu (\dif s \, \dif x) = \bar\mu (\dif s) \nu_s (\dif x)$. By the chain rule for relative entropy,
\begin{equ}\label{eq:ent-decomp}
    \Ent_\mu(F^2)
    =\bar\mu\bigl(\Ent_{\nu_s}(F^2)\bigr)
     +\Ent_{\bar\mu}\bigl(\nu_s(F^2)\bigr).
\end{equ}
The first term in \eqref{eq:ent-decomp} is bounded by \eqref{eq:conditionalLSI}:
\begin{equ}\label{eq:first-entropy}
    \bar\mu\bigl(\Ent_{\nu_s}(F^2)\bigr)
    \le2K\sum_{i\in I}\mu(|\nabla_iF|^2).
\end{equ}
Next we bound the second term in \eqref{eq:ent-decomp}. By \eqref{eq:marginalLSI} we know that
\[
    \Ent_{\bar\mu}(f^2)
    \le 4 \sum_{e\in S} \bar\mu (|\nabla_e f |^2)
\]
so we need to bound $\sum_{e\in S}|\nabla_e f |^2$ where $f$ corresponds to $\sqrt{\nu_s(F^2)}$ in \eqref{eq:ent-decomp}. For $\eta>0$, define
\[
    G_\eta(s)=\sqrt{\nu_s(F^2)+\eta}.
\]

\begin{lemma}
There is an absolute constant $C>0$ so that
\begin{equ}\label{eq:gradientCE}
    \sum_{e\in S}|\nabla_eG_\eta|^2
    \le 2\sum_{e\in S} \nu_s(|\nabla_eF|^2)
     + CK t^2 n \sum_{i\in I} \nu_s (|\nabla_i F|^2 ).
\end{equ}
\end{lemma}

\begin{proof}

We start by computing derivatives of $G_\eta$:
\begin{equ}\label{eq:gradG}
	\nabla_e G_\eta = \frac{\nabla_e (\nu_s F^2)}{2\sqrt{\nu_s (F^2) + \eta}} \;.
\end{equ}
We can write
\begin{equ}
	\nu_s (F^2) = \frac{\int F^2 e^{H} \, \dif x}{\int e^{H} \, \dif x}
\end{equ}
where $H(x, s) = N \beta \sum_{p \in \CP^+_L} \Re \Tr (U_p)$. Then
\begin{equs}\label{eq:gradFCov}
	\nabla_e \big(\nu_s (F^2)\big) &= \nu_s \big(\nabla_e (F^2)\big) - \nu_s (F^2 \nabla_e H) + \nu_s (F^2) \nu_s (\nabla_e H) \\
	&= \nu_s \big(\nabla_e (F^2)\big) - \nu_s (F^2 \nabla_e H_R) + \nu_s (F^2) \nu_s (\nabla_e H_R)
\end{equs}
where $R$ is the block, or union of two blocks if $e$ lies at the boundary of a block, containing $e$, and $H_R = -N \beta \sum_{p \subset R} \Re \Tr (U_p)$. 

We can decompose $\nu_s$ as $\nu_{R^c, s} \otimes \nu_{R, s}$ where the first measure contains the product over all blocks not in $R$, and the second contains the product over all blocks in $R$. Then
\begin{equ}\label{eq:Rsplit}
	\nu_s (F^2) \nu_s (\nabla_e H_R) = \nu_{R^c, s} \big(\nu_{R, s} (F^2)\big) \nu_{R^c, s} \big(\nu_{R, s} ( \nabla_e H_R)\big) = \nu_{R^c, s} \big(\nu_{R, s} (F^2) \nu_{R, s} (\nabla_e H_R)\big)
\end{equ}
where the last equality follows because $\nabla_e H_R$ is independent of the variables in $R^c$. We thus have
\begin{equs}
	&\lvert \nu_{R, s} (F^2 \nabla_e H_R) - \nu_{R, s} (F^2) \nu_{R, s} (\nabla_e H_R) \rvert \\
	&= \Big\lvert \nu_{R, s} \Big(\big(F - \nu_{R, s} (F)\big) \big(F + \nu_{R, s} (F)\big) \big(\nabla_e H_R - \nu_{R, s} (\nabla_e H_R)\big)\Big) \Big\rvert \\
	&\leq 2\| \nabla_e H_R \|_\infty \nu_{R, s} \Big(\big\lvert \big(F - \nu_{R, s} (F)\big) \big(F + \nu_{R, s} (F)\big) \big\rvert\Big) \\
	&\leq 4\| \nabla_e H_R \|_\infty \mathrm{Var}_{\nu_{R, s}} (F)^{1/2} \nu_{R, s} (F^2)^{1/2}.
\end{equs}
Combining \eqref{eq:gradFCov} and \eqref{eq:Rsplit}, we obtain
\begin{equs}
	\lvert \nabla_e (\nu_s (F^2)) \rvert &\leq 2 \nu_s (F^2)^{1/2} \nu_s (\lvert \nabla_e F \rvert^2)^{1/2} + 4\| \nabla_e H_R \|_\infty \nu_{R^c, s} \big(\mathrm{Var}_{\nu_{R, s}} (F)^{1/2} \nu_{R, s} (F^2)^{1/2}\big) \\
	&\leq 2 \nu_s (F^2)^{1/2} \nu_s (\lvert \nabla_e F \rvert^2)^{1/2} + 4\sqrt{K} \| \nabla_e H_R \|_\infty \nu_{R^c, s}\big( \nu_{R, s} (\lvert \nabla_x F \rvert^2)^{1/2} \nu_{R, s} (F^2)^{1/2})\big) \\
	&\leq 2 \nu_s (F^2)^{1/2} \nu_s (\lvert \nabla_e F \rvert^2)^{1/2} + 4\sqrt{K} \| \nabla_e H_R \|_\infty \nu_s (\lvert \nabla_x F \rvert^2)^{1/2} \nu_s (F^2)^{1/2}
\end{equs}
where $\nabla_x$ is the gradient in the directions of $I_R$. In the first and third inequalities above we apply Cauchy--Schwarz, and in the second we use the Poincar\'e inequality \eqref{eq:blockPoincare} for $\nu_{R, s}$. Applying this bound in \eqref{eq:gradG} yields
\[
	\lvert \nabla_e G_\eta \rvert \leq  \nu_s (\lvert \nabla_e F \rvert^2)^{1/2} +2\sqrt{K} \| \nabla_e H_R \|_\infty \nu_s (\lvert \nabla_x F \rvert^2)^{1/2} \;.
\]
Squaring both sides,
\[
	\lvert \nabla_e G_\eta \rvert^2 \leq  2 \nu_s (\lvert \nabla_e F \rvert^2) + 8K \| \nabla_e H_R \|^2_\infty \nu_s (\lvert \nabla_x F \rvert^2) \;.
\]
Summing over $e$ and using that $\| \nabla_e H_R \|_\infty \leq 2 N^{3/2} \lvert \beta \rvert \leq 2t$ by a similar calculation to \eqref{eq:gradW}, we find
\[
	\sum_{e \in S} \lvert \nabla_e G_\eta \rvert^2 \leq 2 \sum_{e \in S} \nu_s (\lvert \nabla_e F \rvert^2) + 32K t^2 (n+1) \sum_{i \in I} \nu_s (\lvert \nabla_i F \rvert^2)
\]
where the $n+1$ arises from the fact that each $\nabla_i F$ term shows up once for each $e \in S$ in the same block as $i$. 

\end{proof}

Since $F^2$ and the density of $\nu_s$ both depend smoothly on $s$, we can take a limit of this bound as $\eta \to 0$ and obtain that \eqref{eq:gradientCE} also holds if $G_\eta (s)$ is replaced by $\sqrt{\nu_s (F^2)}$. Thus, we obtain
\begin{equ}\label{eq:second-entropy}
    \Ent_{\bar\mu}\bigl(\nu_s(F^2)\bigr)
    \le8\sum_{e\in S}\mu(|\nabla_eF|^2)
      +CK t^2 n \sum_{i\in I}\mu(|\nabla_iF|^2).
\end{equ}
Combining \eqref{eq:ent-decomp}, \eqref{eq:first-entropy}, and \eqref{eq:second-entropy}, we obtain Theorem~\ref{thm:main}.

\section{Discussion on dependence of the LSI constant on $\beta$}

In Theorem~\ref{thm:main} our LSI constant $C_{N,\beta}\to \infty$ as $\beta\to\infty$. While our constant is certainly not optimal, this divergence does not appear to be merely a technical condition.
We do not expect, in general, that the LSI holds with a constant uniform in $\beta$. 
To see why, consider a probability measure $\mu$ on any Lie group with Haar density $\frac{1}{Z} e^{-\beta H}$. Suppose $\mu$ has a critical point $\hat x$ which is not a global minimizer and has vanishing Hessian:
\begin{equ}[eq:flat-saddle]
H(\hat x) > \min H,\qquad H(x)=H(\hat x)+O(|x-\hat x|^3)
\end{equ}
when $x$ is in a neighborhood of $\hat x$. For an elementary example on $\R$, consider $\frac1Z e^{-\beta (x^3-1)^2} dx$ with $\hat x=0$. 
We claim that this forces $C_{\mathrm{LS}}\ge c\beta^{1/3}$. Here and in what follows, all constants are independent of $\beta$ provided $\beta$ is sufficiently large, and the value of constants $c$ and $C$ may change from line to line.

Indeed, consider a ball $B_\beta=B(\hat x,r)$ with small radius $r=\beta^{-1/3}$. Let
$f$ be a smooth function such that $f=1$ in $B(\hat x,r/2)$, 
$\supp f\subset B(\hat x,r)$, and $|\nabla f| \lesssim 1/r$.
The last condition is easily achieved since $f$ only has to decay from $1$ to $0$ as the radius increases from $r/2$ to $r$.
Then there exists $K>0$ such that
\[
\beta \sup_{B_\beta} |H-H(\hat x)| \le K \beta r^3 \le K.
\]
From this we see that 
\[
\mu(f^2) \ge Z^{-1} e^{-\beta H(\hat x)-K} |B(\hat x,r/2)|,
\qquad
\mu(B_\beta) \le Z^{-1} e^{-\beta H(\hat x)+K} |B(\hat x,r)|,
\]
and $|B(\hat x,r/2)| / |B(\hat x,r)|\ge c_0>0$. Combining these yields $\mu(f^2)\ge c\mu(B_\beta)$. Additionally, the bound on $\lvert \nabla f \rvert$ implies
\[
\mu( |\nabla f|^2 ) \leq C r^{-2} \mu(B_\beta)
\]
and so 
\[
\mu(f^2) \ge c r^2 \mu( |\nabla f|^2 ) .
\]
On the other hand, due to the non-minimizing energy at $\hat x$, $B_\beta$ has small probability; it is easy to show that
\[
\log (1/\mu (B_\beta))\ge  c \beta.
\]
Now we apply Jensen's inequality to the convex function $x\mapsto x\log (x/m)$
with $m=\E[f^2]$, noting that 
noting that $\E[f^2 | B_\beta]=m/p$ with $p=\mu(B_\beta)$:
\begin{equs}
\Ent_\mu (f^2)
&=\E\Big[f^2 \log \frac{f^2}{m}\Big] 
= p\E\Big[f^2 \log \frac{f^2}{m}|B_\beta\Big] 
\ge p \E[f^2 |B_\beta] \log\frac{\E[f^2 |B_\beta] }{m} \\
&= \mu(f^2) \log \big(1/\mu (B_\beta)\big)
\ge c\beta r^2 \mu( |\nabla f|^2 )
 = c\beta^{1/3} \mu( |\nabla f|^2 ) .
\end{equs}
For Yang--Mills with density $e^{-\beta N^2 H(U)}$, we have the potential
\[
H(U)=\sum_p V(U_p),\qquad
V(g) =1-\frac1N \Re\Tr g.
\]
We can construct fields $\hat U$ satisfying \eqref{eq:flat-saddle}. For instance, let $G=\mathrm{SU}(4)$ and $I_4$ be the identity, and define $\hat U=(\hat U_1,\hat U_2)$ at edges starting from $(x,y)\in \Lambda_L$ to be
\[
\hat U_1 (x,y)=i^{-y} I_4 \in \mathrm{SU}(4),
\qquad
\hat U_2 (x,y)=I_4 \in \mathrm{SU}(4),
\]
so that 
\[
\hat U_p = \hat U_1 (x,y) \hat U_1 (x,y+1)^{-1}
= i I_4,
\qquad \forall p.
\]
Note that the minimizer of $V$ is actually $I_4$, at which $V=0$. For $X \in \mathfrak{su}(4)$, since $\Tr(X^2)\in \R$, we have
\[
V( i I_4 e^{tX})=1-\frac1N \Re\Tr(i e^{tX})=1+O(t^3\|X\|^3).
\]
From this one can derive that $H(U)=H(\hat U)+O(d(U,\hat U)^3)$ in a neighborhood of $\hat U$, by a similar calculation as \eqref{eq:dW} where one takes the variation of each edge variable $U_e$ in the direction $X_e\in \mfg$, and uses the fact that in a neighborhood of $\hat U$ one has $d(U,\hat U)=\| (X_e)_e \|$. This is precisely \eqref{eq:flat-saddle}, so the above general argument  shows that $C_{\mathrm{LS}}\ge c\beta^{1/3}$ by
 constructing functions $F$ supported in an $r$-neighborhood of $\hat U$.  

Clearly one can take the same example when $N$ is any multiple of $4$. Note that \eqref{eq:flat-saddle} is only one mechanism by which the LSI constant blows up as $\beta\to \infty$, and in general there could be many other possible scenarios; we do not pursue this further here.
 
\begin{remark}
One can also choose  $F$ to be gauge invariant. One only needs to replace $\hat U$ by its gauge orbit $\hat O = \{g\cdot \hat U:g\in \mathcal G\}$ where $\mathcal G$ is the gauge group,
and use the distance  $\delta(U)=\min_{g\in \mathcal G} d(U,g\cdot \hat U)$.
For a smooth cutoff function $\chi$ supported on $[0,1]$ which equals $1$ on $[0,1/4]$, define $F(U)=\chi(\delta(U)^2/r^2)$ and a tubular neighborhood $B_\beta$ using the distance $\delta$. As a general fact, the square of the distance to a smooth submanifold 
is always a smooth function in a tubular  neighborhood of this submanifold. It is easy to check that $\delta$ is gauge invariant (that is, $\delta(g\cdot U)=\delta(U)$ for $g \in \CG$), so $F$ is gauge invariant.
\end{remark}

\begin{remark}
The continuum analogue of the above example is as follows.  Let $\mathbf T^2$ be the unit torus.
Take a matrix $B=\mbox{diag}(1,1,1,-3)\in \mfg$. 
On $[0,1]^2$ (not yet $\mathbf T^2$), define a $\mathfrak{g}$-valued connection $A=8\pi i y B dx$. Note that 
$A$ has different values at the boundaries $y=0$ and $y=1$; however, since one can glue the two boundaries by an $\mathrm{SU}(4)$-valued transition function $g(x)=e^{8\pi i x B}$ so that 
\[
A(x,1)=g(x)^{-1} A(x,0) g(x)+ g^{-1}(x) dg(x),
\]
$A$ is indeed a well-defined connection.
The curvature of $A$ is $F_A=-8\pi i B dx\wedge dy$.
Writing $d_A$ for the covariant derivative, we have
\[
d_A^* F_A = d_A (*F_A)=0
\]
since $F_A$ is constant and $[A,*F_A]=0$. In other words, $A$ is a critical point of the Yang--Mills action. If we ``discretize'' $\mathbf T^2$ into a $4\times 4$ lattice, then along a horizontal edge  $e$ (whose length is $1/4$) starting from $(j/4,k/4)$
\[
U_1(j,k)= \mbox{hol}_e(A)
= \exp\Big(- \frac14\cdot  8\pi i \frac{k}4  B \Big)
= e^{-i\pi k B/2}=i^{-k} I_4.
\]
The connection $A$ lives on the $\mathrm{SU}(4)$ bundle $E$,
which splits as $E=L_4\oplus L_4\oplus L_4\oplus L_{-12}$
where $L_m$ stands for the line bundle with degree $m$; recall that $\deg L= \frac{i}{2\pi} \int_{\mathbf T^2}F$ and $\frac{i}{2\pi}(-8\pi i)(1,1,1,-3)=(4,4,4,-12)$. Constructing such critical points by splitting or reducing to line bundles is a standard procedure, c.f. \cite[above (4.35)]{Witten2revisited}, \cite[Section~6]{AB83}.
\end{remark}

\bibliographystyle{./Martin}
\bibliography{./refs}

\end{document}